\documentclass[a4paper]{article}
\usepackage{amsmath,amsfonts,amsthm, amssymb,xspace}
\usepackage{latexsym, color, graphicx}
\def\H{{\cal H}}
\def\P{{\cal P}}
\def\F{{\cal F}}

\newcommand{\N}{{\mathbb N}}

\newcommand{\WP}{\textsf{WP}}
\newcommand{\GWP}{\textsf{GWP}}
\newcommand{\eda}{\textsf{eda}\xspace}
\newcommand{\edas}{\textsf{eda}s\xspace}

\newcommand{\GIB}{\mathrm{GIB}}

\newcommand{\Cay}{\Gamma}
\newcommand{\HatCay}{\widehat{\Cay}}
\newcommand{\Sch}{\Sigma}
\newcommand{\IdC}{1_{\Cay}}
\newcommand{\IdHatCay}{1_{\HatCay}}
\newcommand{\IdS}{1_{\Sch}}
\newcommand{\start}{{}S}
\newcommand{\nullstring}{\varepsilon}
\newcommand{\fsa}{{\sf fsa}\xspace}
\newcommand{\pda}{{\sf pda}\xspace}
\newcommand{\gsm}{{\sf gsm}\xspace}

\newtheorem{theorem}{Theorem}[section]
\newtheorem{property}{Property}[section]
\newtheorem{proposition}[theorem]{Proposition}

\newenvironment{mylist}{\begin{list}{}{
\setlength{\parskip}{0mm}
\setlength{\topsep}{2mm}
\setlength{\parsep}{0mm}
\setlength{\itemsep}{0.5mm}
\setlength{\labelwidth}{7mm}
\setlength{\labelsep}{3mm}
\setlength{\itemindent}{0mm}
\setlength{\leftmargin}{12mm}
\setlength{\listparindent}{6mm}
}}{\end{list}}
\title{The generalised word problem in hyperbolic and relatively
hyperbolic groups}
\author{Laura Ciobanu, Derek Holt, and Sarah Rees}
\begin{document}

\maketitle
\begin{abstract}
We prove that, for a finitely generated group hyperbolic relative to virtually
abelian subgroups, the generalised word problem for a parabolic subgroup is
the language of a real-time Turing machine.  Then, for a hyperbolic group, we
show that the generalised word problem for a quasiconvex subgroup is a real-time
language under either of two additional hypotheses on the subgroup. 
 
By extending the Muller-Schupp theorem we show that the generalised word
problem for a finitely generated subgroup of a finitely generated virtually
free group is context-free. Conversely, we prove that a hyperbolic group must
be virtually free if it has  a torsion-free quasiconvex subgroup
of infinite index with context-free generalised word problem.
\end{abstract}


\section{Introduction}
\label{sec:intro}
Let $G=\langle X \rangle$ with $|X|<\infty$ be a group and $H\leq G$.
The {\em word problem} $\WP(G,X)$ and {\em generalised word problem}
$\GWP(G,H,X)$ are defined to be the preimages $\phi^{-1}(\{1_G\})$ and
$\phi^{-1}(H)$ respectively, where $\phi$ is the natural map from the set of
words over $X$ to $G$.
We are interested in the relationship between the algebraic properties
of $G$ (and $H$) and the formal language classes containing $\WP(G,X)$ and
$\GWP(G,H,X)$. These questions have already been well studied for the word
problem, but relatively little for the generalised word problem.  Since, as is
well known for $\WP(G,X)$, the question of membership of
$\WP(G,X)$ and $\GWP(G,H,X)$ in a formal language family $\F$ is typically
independent of the choice of
the finite generating set $X$,
we shall usually use the simpler notations $\WP(G)$ and $\GWP(G,H)$.  It will
be convenient to assume throughout the paper that all generating sets $X$ of
groups $G$ are closed under inversion; i.e. $x \in X \Rightarrow x^{-1} \in X$.

It is elementary to prove that $\WP(G)$ is regular if and only if $G$ is
finite and, more generally, that $\GWP(G,H)$ is regular if and only if $|G:H|$
is finite. It is well-known that $\WP(G)$ is context-free if and only if $G$
is virtually free \cite{MullerSchupp83}, and it is shown in
\cite{Holt,HoltRees} that $\WP(G)$ is a {\em real-time} language
(that is, the language of a real-time Turing machine) for several
interesting classes of groups, including hyperbolic and geometrically finite
hyperbolic groups. The solvability of $\GWP(G,H)$ has been established for numerous classes of groups, most recently for all (compact and connected) $3$-manifold groups \cite{FriedlWilton}.

In this paper, we study the conditions under which $\GWP(G,H)$ is a real-time or
a context-free language for subgroups $H$ of hyperbolic and relatively
hyperbolic groups. There are several definitions of relatively hyperbolic groups, and the one
that we are using here is that of \cite{Osin06}; so in particular the
Bounded Coset Penetration Property holds. Our first result is the following.

\begin{theorem}
\label{thm:relhyp}
Suppose that the finitely generated group $G$ is hyperbolic relative to a set
$\{H_i: i \in I\}$ of virtually abelian (parabolic) subgroups of $G$,
and that $H$ is a selected parabolic subgroup.
Then $\GWP(G,H)$ is a real-time language.
\end{theorem}

The proof uses a combination of results for relatively hyperbolic groups
that were developed by Antol\'in and Ciobanu in \cite{AntolinCiobanu} and the
{\em extended Dehn algorithms} that were introduced by Goodman and
Shapiro in \cite{GoodmanShapiro}.  By choosing $H=H_0$ to be the trivial
subgroup of $G$, we obtain a generalisation of the result proved 
(also using extended Dehn algorithms) in \cite{HoltRees} 
that $\WP(G)$ is a real-time language when $G$ is a geometrically finite
hyperbolic group. 

Since Goodman and Shapiro's techniques, such as the $N$-tight Cannon's algorithm
and related results (Theorem 37 in \cite{GoodmanShapiro}), only apply to virtually abelian groups and not, for example, to all virtually nilpotent groups,
our proof cannot be extended to parabolics beyond those that are virtually abelian.
Furthermore, even with a different approach to the proof there must be some
limitations on the choice of parabolics, since examples exist
 of relatively hyperbolic groups with
generalised word problems that are not real-time.
If one considers, for example, the free product $G=H \ast K$ of two groups $H$
and $K$, where $H$ is hyperbolic and $K$ has unsolvable word problem,
then $G$ is hyperbolic relative to $K$, but it is easy to see that the subgroup membership problem for $K$ in $G$ is unsolvable, so it cannot be real-time.

Recall that a subgroup $H \le G$ is called {\em  quasiconvex} in $G$ if
geodesic words over $X$ that represent elements of $H$
lie within a bounded distance of $H$ in the Cayley graph $\Cay(G,X)$.
And $H$ is called {\em almost malnormal} in $G$ if
$|H \cap H^g|$ is finite for all $g \in G \setminus H$.
We conjecture that, for a hyperbolic group $G$, the set $\GWP(G,H)$ is a real-time
language for any quasiconvex subgroup $H$ of $G$, but we are currently only
able to prove this under either of two 
additional hypotheses:

\begin{theorem}
\label{thm:qcsubhyp}
Let $G$ be a hyperbolic group, and
$H$ a quasiconvex subgroup of $G$. Suppose that either
\begin{mylist}
\item[(i)] $H$ is almost malnormal in $G$; or\\
\item[(ii)] $|C_G(h):C_H(h)|$ is finite for all $1 \ne h \in H$.
\end{mylist}
Then $\GWP(G,H)$ is a real-time language.
\end{theorem}

Our proofs of the result in the two cases are quite distinct, and so we write
them separately.  Under assumption (i), $G$ is
hyperbolic relative to $\{H\}$ \cite[Section 1, Example (III)]{Osin06}, and
the proof is similar to that of Theorem \ref{thm:relhyp} (although $H$ is not
usually virtually abelian, so we cannot apply that result directly).  Under
assumption (ii), we make use of some results of Foord \cite{Foord} about the
Schreier graph of $G$ with respect to $H$. Note that, since
centralisers of elements of infinite order in hyperbolic groups are
virtually cyclic, $|C_G(h):C_H(h)|$ is always finite for such elements
$h$ and so, in particular, assumption (ii) holds whenever $G$ is torsion-free.

Our next result, which concerns context-free generalised word problems, is
straightforward to prove and may be known already, but it does not appear
to be in the literature.

\begin{theorem}
\label{thm:subvf}
Let $G$ be finitely generated and virtually free, and let $H$ be a finitely
generated subgroup of $G$. Then $\GWP(G,H)$ is deterministic context-free.
\end{theorem}

The conclusion of Theorem~\ref{thm:subvf} may or may not hold if we drop the
condition that $H$ is finitely generated. Suppose that $G$ is the free group
on two generators $a,b$. For $H_1=[G,G]$, the set
$\GWP(G,H_1)$ consists of all words whose exponent sums in both $a$ and $b$
are zero, and is not context-free, but for the subgroup $H_2$ of words whose
exponent sum in $a$ is zero, the set $\GWP(G,H_2)$ is context-free.

We would like to know to what extent Theorem~\ref{thm:subvf} is best possible
when $H$ is finitely generated.
We observe that, where $H_G := \cap_{g \in G} H^g$ is the {\em core} of $H$ in
$G$, the set $\GWP(G,H,X)$ is the same set of words as $\GWP(G/H_G,H/H_G,X)$. 
We know of no examples for which $H$ is finitely generated with trivial core
and $\GWP(G,H)$ is context-free, but $G$ is not virtually free.
The following result is an attempt at a converse to Theorem~\ref{thm:subvf}.

\begin{theorem}
\label{thm:subvf_conv}
Let $G$ be a hyperbolic group, and let $H$ be a quasiconvex subgroup of
infinite index in $G$ such that $|C_G(h):C_H(h)|$ is finite for all
$1 \ne h \in H$. If $\GWP(G,H)$ is context-free then $G$ is virtually free.
\end{theorem}

Note that we make the same assumption on centralisers of elements $h \in H$
as in Theorem \ref{thm:qcsubhyp} (ii), and again we conjecture that this
is not necessary for the conclusion of the theorem.
However the necessity of quasiconvexity is demonstrated by a
construction found in \cite{Rips}, as follows.
Let $Q$ be any finitely presented group, and
choose $\lambda >0$. We can define a finitely presented group $G$ and a normal
$2$-generated subgroup $H$ of $G$, such that $G/H \cong Q$, and $G$ satisfies
the small cancellation condition $C'(\lambda)$; in particular we can choose
$Q$ with insoluble word problem, and choose $\lambda \leq 1/6$ to ensure
that $G$ is hyperbolic, and in that case $\GWP(G,H)$ is not even recursive.

The closure properties of context-free languages
\cite[Chapter 11]{HopcroftUllman} and of real-time languages \cite{Rosenberg}
ensure that all of the above results are independent of the choice of the finite
generating set $X$ of $G$, so we are free to choose $X$ to suit our own
purposes in the proofs. (More generally, for membership of $\WP(G,X)$ or
$\GWP(G,H,X)$ in a formal language class $\mathcal{F}$ to be independent of
the choice of $X$, we need $\mathcal{F}$ to be closed under inverse
homomorphism. This property holds for all of the most familiar formal
language classes, including regular, context-free, deterministic
context-free, real-time, context-sensitive, deterministic context-sensitive,
and recursive languages.) 

This article is structured as follows.
In Section~\ref{sec:relhypgps}, we summarise the basic properties of relatively
hyperbolic groups that we shall need, and we 
recall some of their properties that are proved in \cite{AntolinCiobanu}.
In Section \ref{sec:eda}, we introduce the concept of
{\em extended Dehn algorithms} for solving the word and generalised
word problems in groups and we recall some results pertaining to relatively
hyperbolic groups that are proved in \cite{GoodmanShapiro}.
Sections~\ref{sec:main_thm}, \ref{sec:qcthmi}, \ref{sec:qcthmii},
\ref{sec:subvf} and \ref{sec:subvf_conv} contain the proofs of Theorems
\ref{thm:relhyp}, \ref{thm:qcsubhyp} (i), \ref{thm:qcsubhyp} (ii),
\ref{thm:subvf} and \ref{thm:subvf_conv}, respectively.
Finally, in Section \ref{sec:foord}, we sketch a proof of a result of
Foord \cite[Theorem 4.3.1.1]{Foord} that we shall need, since the original
source might not be readily available to readers.

\section{Relatively hyperbolic groups}
\label{sec:relhypgps}
We follow \cite{Osin06} for notation and the definition of relatively hyperbolic
groups. This definition is equivalent to what Farb calls
``strong relative hyperbolicity'' in \cite{Farb}.

Suppose that $G$ is a group, $X$ a finite generating set,
and $\{H _i : i \in I\}$ a
collection of subgroups of $G$, which we call {\em parabolic} subgroups. We
define $X_i := X \cap H_i$ to be the set of generators in $X$ that lie within
the subgroup $H_i$, and $X_I := \cup_{i \in I} X_i$.  Then we define $\H$ and
$\widehat{X}$ as the sets
\[ \H := \bigcup_{i \in I} (H_i \setminus \{1\}),\quad 
\widehat{X} = X \cup \H. \]

Much of our argument involves the comparison of lengths of various words
that represent an element $g \in G$ written over different sets, namely
$X$, $\widehat{X}$, a set $Z \supseteq X$ that is introduced during the extended
Dehn algorithm (which, when $Z \supset X$, is not actually a generating set
for $G$, since only some of the words over $Z$ correspond to elements of $G$),
and certain subsets of these sets.

So we shall consider the Cayley graphs $\Cay$ and $\HatCay$ for $G$
over the generating sets $X$ and $\widehat{X}$,
and view words over $X$ and $\widehat{X}$ also as paths in $\Cay$ and
$\HatCay$. We denote by $d_\Cay$ the graph distance in $\Cay$ and
by $d_{\HatCay}$ the graph distance in $\HatCay$. For a word
$w$ written over a set $Y$ (that is, an element of $Y^*$), we write $|w|$ to
denote the length of $w$
and, for a group element $g$, we write $|g|_Y$ to denote the length of a
shortest word over $Y$ that represents the group element $g$ (assuming that
such a word exists). We call $|g|_Y$ the {\em $Y$-length} of $g$, and a
shortest word over $Y$ that represents $g$ a {\em $Y$-geodesic} for $g$.

Most words will be written over either $X$ or $\widehat{X}$ and, in order to
make a clear distinction between those two types of words we will normally use
Roman letters as names
for words over $X$ and paths in $\Cay$, and Greek letters as names for words
over $\widehat{X}$ and paths in $\HatCay$, with the exception that we will
write $\widehat{w}$ for the word over $\widehat{X}$ that is derived from a word $w$
over $X$ by a process called {\em compression}, which will be described
later in this section.

We refer to \cite{Osin06} for a precise definition of relative hyperbolicity of
$G$ with respect to $\{H_i: i \in I\}$. Under that definition, relative
hyperbolicity is known to be equivalent to the fact that the Cayley graph
$\HatCay$ is $\delta$-hyperbolic
for some $\delta$ together with the Bounded Coset Penetration Property, stated
below as Property \ref{prop:BCP}.

From now on we shall assume that $G$ is relatively hyperbolic in this sense.
It is proved in \cite{Osin06} that, under our assumption that $G$ is finitely
generated, the subgroups $H_i$ are finitely generated, the set $I$ 
is finite, and any two distinct parabolic subgroups have finite intersection. 
We assume that the generating set $X$ is chosen such that
$H_i = \langle X_i \rangle$ for all $i \in I$.

We need some terminology relating to paths in $\HatCay$.
\begin{enumerate}
\item
We call a subpath of a path $\pi$ an {\em $H_i$-component}, or
simply a {\em component}, of $\pi$ if it is written as a word over $H_i$ for
some $i \in I$, and is not contained in any longer such subpath of $\pi$.
\item
Two components of (possibly distinct) paths
are said to be {\em connected} if both are
$H_i$-components for some $i\in I$, and both start within the same left coset
of $H_i$.  
\item
A path is said to {\em backtrack} if it has a pair of
connected components.
A path is said to {\em vertex-backtrack} if it has a subpath of length
greater than 1 that is labelled by a word representing an element of some $H_i$.

Note that if a path does not vertex-backtrack, then it does not backtrack and
all of its components are edges.
\item
For $\kappa \ge 0$ we say that two paths
are {\em $\kappa$-similar} if the $X$-distances between their two initial
vertices and between their two terminal vertices are both at most $\kappa$.
\end{enumerate}

\begin{property}
[Bounded Coset Penetration Property {\cite[Theorem 3.23]{Osin06}}]
\label{prop:BCP}
For any $\lambda \geq 1, c \geq 0, \kappa \geq 0$, there exists a constant
$\epsilon = \epsilon (\lambda, c, \kappa)$ such that, for any two
$\kappa$-similar $(\lambda, c)$-quasi-geodesic
paths $\pi$ and $\pi'$ in $\HatCay$ that do not backtrack,
the following conditions hold:
\begin{enumerate}
\item[{\rm(1)}] The sets of vertices of $\pi$ and $\pi'$ are contained in the
closed $\epsilon$-neighbourhoods (with respect to the metric $d_\Cay$) of
each other.
\item[{\rm(2)}] For any $H_i$-component $\sigma$ of $\pi$ for which the $X$-distance between its endpoints is greater than $ \epsilon$,
some $H_i$-component $\sigma'$ of $\pi'$ is connected to $\sigma$.
\item[{\rm(3)}] Whenever $\sigma$ and $\sigma'$ are connected $H_i$-components of $\pi$ and $\pi'$ respectively, 
the paths $\sigma$ and $\sigma'$ are $\epsilon$-similar.
\end{enumerate}
\end{property}

We shall need to use two more properties of relatively hyperbolic groups
that are proved in \cite{AntolinCiobanu}.

\begin{property}[{\cite[Theorem 5.2]{AntolinCiobanu}}] \label{prop:AC1}
Let $Y$ be a finite generating set for $G$. Then for some $\lambda \geq 1, c
\geq 0$ there exists a finite set $\Psi$ of non-geodesic words over
$Y\cup \H$ such that:
\begin{quote}
every 2-local geodesic word over $Y\cup \H$ not containing any element of
$\Psi$ as a subword labels a $(\lambda,c)$-quasi-geodesic path in
$\HatCay$ without vertex-backtracking. 
\end{quote}
\end{property}

Now suppose that $v$ is any word in $X^*$.
Then following \cite[Construction 4.1]{AntolinCiobanu} we define $\widehat{v}$
to be the word over $\widehat{X}$
that is obtained from $v$ by replacing (working from the left) each
subword $u$ that is maximal as a subword over some $X_i$ ($i \in I$) by
the element $h_u$ of $\H$ that the subword represents.
We call these $X_i$-subwords $u$ of $v$ its {\em parabolic segments},
and use the term {\em compression} for the process that converts $v$ to
$\widehat{v}$. A word $v$ is said to have {\em no parabolic shortenings} if each
of its parabolic segments is an $X_i$-geodesic.

In order to avoid confusion we comment that a parabolic segment (which is a
maximum 
subword over some $X_i$ of a word over $X$) is not quite the same as a component
(which is a maximal subpath/subword over some $H_i$ of a path/word over
$\widehat{X}$); but clearly the two concepts are close.

The other required property is proved in \cite[Lemma 5.3]{AntolinCiobanu};
the precise description of $\Phi$ is taken from the proof of that lemma, rather
than from its statement.

\begin{property}\label{prop:AC2}
Let $Y$ be a finite generating set for $G$. Then for some
$\lambda  \geq 1$,
$c\geq 0$, some finite subset 
$\H'$ of $\H$, and any finite generating set $X$ of $G$ 
with $$Y\cup \H'\subseteq X\subseteq Y\cup \H,$$ there is a finite subset
$\Phi$ of non-geodesic words over $X$ such that:
\begin{quote}
if a word $w \in X^*$ has no parabolic shortenings and no subwords in $\Phi$,
then the word $\widehat{w} \in \widehat{X}^*$ is a 2-local geodesic and labels a
$(\lambda, c)$-quasi-geodesic path in $\HatCay$ without vertex-backtracking.

Furthermore, for every $i\in I$ and $h\in H_i$, we have $|h|_{X}=|h|_{X_i}$.
\end{quote}
In fact $\Phi = \Phi_1 \cup \Phi_2$, where $\Phi_1$ is the set of non-geodesic
words in $X^*$ of length 2, and $\Phi_2$ is the set of all words $u \in X^*$
with no parabolic shortening and for which $\widehat{u} \in \Psi$,
where $\Psi$ (together with $\lambda$ and $c$) is given by
Property~\ref{prop:AC1}.
\end{property}

\section{Extended Dehn algorithms}
\label{sec:eda}
Our proofs of Theorem~\ref{thm:relhyp} and both parts of
Theorem~\ref{thm:qcsubhyp} depend on the construction of an
{\em extended Dehn algorithm} (\eda) \cite{GoodmanShapiro} for $G$ with
respect to $H$.  In each case, we then need to show that the \eda satisfies a
particular condition that allows us to apply Proposition~\ref{prop:realtime}
(below)
in order to verify both that the algorithm solves $\GWP(G,H)$ and that it
can be programmed on a real-time Turing machine.
Proposition~\ref{prop:realtime} is derived from \cite[Theorem 4.1]{HoltRees},
which was used to prove the solubility of the word problem in
real-time for various groups with \edas to solve that problem. We restate that 
result as a proposition in this paper for greater clarity of exposition.

Our definition of an extended Dehn algorithm (which is defined with respect to
a specific finite generating set $X$ of $G$) is modelled on the
definition of \cite{GoodmanShapiro} (where it is called a {\em Cannon's
algorithm}), with the difference that
we are using our algorithm to solve a generalised word problem $\GWP(G,H,X)$
rather than a word problem $\WP(G,X)$.  Elsewhere in the literature
\cite{HoltRees} the same concept is called a {\em generalised Dehn algorithm};
our decision to introduce a new name is based on both our recognition that
there are many various different algorithms attributed to (more than one)
Cannon, and our desire to avoid overuse of the term `generalised'.

For a (Noetherian) rewriting system $R$ with alphabet $Z$ and $w \in Z^*$, we
write $R(w)$ for the reduction of the word $w$ using the rules of $R$. In
general, $R(w)$ may depend on the order in which the rules are applied, and
we shall specify that order shortly. A word $w$ is called {\em ($R$-)reduced}
if $R(w)=w$; that is, if $w$ does not contain the left hand side
of any rule as a subword.

We define an \eda for a finitely generated group $G= \langle X \rangle$
with respect to a subgroup $H \le G$ to be
a finite rewriting system $S$ consisting of rules $u \rightarrow v$, where
\begin{mylist}
\item[(i)]  $u,v \in Z^*$ for some
finite alphabet $Z \supseteq X \cup \{H\}$;
\item[(ii)] $|u|>|v|$; and
\item[(iii)] either $u,v \in (Z \setminus \{H\})^*$
or $u=Hu_1,v=Hv_1$, with $u_1,v_1 \in (Z \setminus \{H\})^*$.
\end{mylist} 
We say that the \eda $S$ solves the generalised word problem $\GWP(G,H,X)$ if,
for every word $w$ over $X$, we have $S(Hw) = H$ if and only if $w$ represents
an element of $H$. If $H=\{1\}$ (in which case we may assume only that
$Z \supseteq X$), then we call $S$ an \eda for $G$.

As observed earlier, for $S(w)$ to be well-defined, we need to specify the
order in which the reduction rules are applied to words $w$.
In the terminology of \cite[Section 1.2]{GoodmanShapiro}, $S$ with the order
we specify below is an
{\it incremental rewriting algorithm} and, since we shall only apply
it to words of the form $Hw$ with $w \in X^*$, the rules $Hu_1 \to Hv_1$
are effectively {\em anchored} rules.
We assume that no two distinct rules have the same left hand sides.
Then we require that when a word $Hw$ contains several left hand sides of
$S$, the rule which is applied is one that ends closest to the start of $Hw$;
if there are several such rules, the one with the longest left hand side is
selected.

In our applications, the rules of the form $u \to v$ with
$u,v \in (Z \setminus \{H\})^*$ will form an \eda $R$ 
for $G$, of a type that is considered in
\cite{GoodmanShapiro}, and which solves the word problem $\WP(G,X)$;
i.e. $R(w)$ is the empty word if and only if $w=_G 1$.
The properties of the \eda $R$ that we shall use are
described in more detail in Proposition \ref{prop:P(D,E)} below.
Many of the technical results of \cite{GoodmanShapiro} apply without
modification to \edas that solve a $\GWP$ rather than $\WP$.

Note that the set $Z$ may properly contain $X \cup \{H\}$, and so contain
symbols that do not correspond to either elements or subsets of $G$.
But it is a consequence of \cite[Proposition 3]{GoodmanShapiro} that a word
in $(Z \setminus \{H\})^*$ that arises from applying these rules to a word
$w \in X^*$ unambiguously corresponds to the element of $G$ represented
by $w$, and so we may interpret such words as elements in $G$.
We shall see shortly that our rewrite rules $Hu \to Hv$ will all be of the
form $Hz \to H$ for words $z \in (Z \setminus \{H\})^*$ that represent
elements of $H$, and so words derived by applying rules of the \eda to $Hw$
with $w \in X^*$ unambiguously represent the coset of $H$ in $G$ defined by
$Hw$.

Following \cite[Corollary 27]{GoodmanShapiro}, for a positive integer $D$, we
say that an \eda $S$ that solves the word problem for a group
$K=\langle Y \rangle$ is {\em $D$-geodesic} if a word $w$ that is $S$-reduced
and represents an element $g$ of $Y$-length at most $D$ must
in fact be written over $Y$, and be a $Y$-geodesic for $g$.

In this section so far, we have not made any assumptions on $G,H,X$, beyond
the finiteness of the generating set $X$.
Suppose now that the parabolic subgroups $H_i$ are all virtually abelian,
and that $R$ is an \eda for $G$  that solves $\WP(G,X)$.
For integers $D \ge E \ge 0$, we say that $R$ satisfies $\P(D,E)$ if the
following conditions hold.
\begin{mylist}
\item[(1)] For each $i$, $X_i := X \cap H_i$ generates $H_i$, and
the alphabet $Z$ of $R$ has the form $Z = \cup_{i \in I} Z_i \cup X$
with $X_i \subseteq Z_i$.
\item[(2)] For each rule $u \to v$ of $R$, we either have $u,v \in X^*$,
or $u,v \in Z_i^*$ for a unique $i \in I$.
\item[(3)] For each $i \in I$, the rules $u \to v$ with $u,v \in Z_i^*$ form
a $D$-geodesic \eda $R_i$ that solves $\WP(H_i,X_i)$.
\item[(4)] All $R$-reduced words $w \in X^*$ that have length at most $E$
are $X$-geodesic.
\end{mylist}

The following result is proved under slightly more general conditions
on the parabolic subgroups in \cite{GoodmanShapiro}, and is stated here in the
form in which we need it:
\begin{proposition}[{\cite[Theorem 37]{GoodmanShapiro}}]
\label{prop:P(D,E)}
Suppose that $G = \langle Y \rangle$ is hyperbolic relative to the virtually
abelian parabolic subgroups $H_i$.
Then, there is a finite generating set $X$ of $G$, consisting of the
generators in $Y$ together with some additional elements from the $H_i$
(that include all non-trivial elements from the intersections $H_i \cap H_j$
with $i \ne j$) with the following property: for all sufficiently large
integers $D,E$ with $D \ge E \ge 0$, there is an $\eda$ for $G$ that solves
$\WP(G,X)$ and satisfies $\P(D,E)$.
\end{proposition}

We need to extend our definition of the property $\P(D,E)$ to our wider
definition of an \eda for a group with respect to a subgroup.
For $G$ satisfying the above hypotheses, and given any subgroup $H \le G$, we
shall say that an \eda $S$ for $G$ with respect to $H$, with alphabet
$Z \cup\{H\}$, satisfies $\P(D,E)$ if the rules in $S$ of the form $u \to v$
with $u,v \in Z^*$ form an \eda $R$ satisfying $\P(D,E)$.

In the proofs of each of Theorems
\ref{thm:relhyp}, \ref{thm:qcsubhyp} (i), \ref{thm:qcsubhyp} (ii),
we shall apply the following result, which is essentially (part of)
\cite[Theorem 4.1]{HoltRees}.
\begin{proposition}
\label{prop:realtime}
Let $G$ be a group, finitely generated over $X$, $H$ a subgroup of $G$, and 
let $S$ be an extended Dehn algorithm for $G$ with respect to $H$.
Suppose that there exists a constant $k$ such that, for any word $w$ over $X$,
we have
\[ |w_1| \leq k \min \{ |g|_X : g\in G,\, g \in  Hw \}, \] 
where $w_1$ is the word over $Z$ defined by $Hw_1=S(Hw)$.
Then $S$ solves the generalised word problem and can be programmed on
a real-time Turing machine.
\end{proposition}
\begin{proof}
Since, for any $w \in \GWP(G,H)$ the minimal length representative of $Hw$
is the identity element, it is
immediate from the inequality that $S$ solves $\GWP(G,H)$.
That an \eda satisfying that condition can be programmed in real-time is 
then an immediate consequence of \cite[Theorem 4.1]{HoltRees}; in fact as stated
that theorem applies only to \edas to solve the word problem, but it is clear 
from the proof that it applies also to $\GWP(G,H)$.
\end{proof}


\section{The proof of Theorem \ref{thm:relhyp}}
\label{sec:main_thm}

Suppose that $G = \langle X \rangle$ satisfies the hypotheses of
Theorem \ref{thm:relhyp}, and that $H=H_0$ is the selected parabolic subgroup,
generated by $X_0 \subset X$.

We start with some adjustments to $X$ that are necessary to ensure that it
satisfies the conditions we need for our arguments. 
These adjustments all consist of appending generators that lie in one of
the parabolic subgroups.
Firstly, we extend $X$ to contain the finite subset $\H'$ of $\H$  defined
in Property \ref{prop:AC2}.
Secondly, we adjoin to $X$ the elements of the $H_i$ that are required by
Proposition \ref{prop:P(D,E)}.
(As stated in Proposition \ref{prop:P(D,E)}, these include all elements in
the finite intersections $H_i \cap H_j$ (for $i \ne j$) of pairs of parabolic
subgroups.)

Associated with this choice of $X$, Properties \ref{prop:AC1} and
\ref{prop:AC2}  specify sets $\Psi$ and $\Phi$
of non-geodesic words over $\widehat{X}$ and $X$ respectively, and
associated parameters $\lambda,c$.
We then define $\epsilon=\epsilon(\lambda,c+1,0)$
to be the constant in the conclusion of Property~\ref{prop:BCP}.

We now apply  Proposition~\ref{prop:P(D,E)} to find an \eda $R$ for
$G$ that solves $WP(G,X)$ and that satisfies $\P(D,E)$ for parameters $D,E$
with $D \ge E$, where $E >\max\{ \epsilon, 2\}$, 
and $E$ is also greater than the length of any word in $\Phi$,
and greater than the $X$-length of any component of any word in
$\Psi$. (The reasons for these conditions will become clear during the proof.)

We can use all properties of $R$ that are proved in
\cite[Section 5]{GoodmanShapiro} and we observe in particular that,
by \cite[Lemma 48]{GoodmanShapiro}, for a word $w \in X^*$, if
$R(w) = u_1vu_2$ where $v$ is a maximal $Z_i$-subword for some $i \in I$,
then there exists $v' \in X_i^*$ with $R_i(v') = v$,
where $R_i$ is the associated \eda for $H_i$. So $v$ unambiguously represents
the element $v' \in H_i$. 

We create an \eda $S$ for $G$ with respect to $H$ by adding to $R$ all
rules of the form $Hz \rightarrow H$ with $z \in Z_0$.
Since none of these new rules actually applies to words over $X$,
it is clear that the \eda $S$ also satisfies $\P(D,E)$.
We shall verify that $S$ solves $\GWP(G,H,X)$, and that it can be programmed
on a real-time Turing machine.

Now suppose that $w \in X^*$ and that $S(Hw) = Hw_1$.
In order to verify that our \eda $S$ solves $\GWP(G,H,X)$
and can be programmed on a real-time Turing machine, it is sufficient by
Proposition~\ref{prop:realtime} to establish the existence of a
constant $k$ that is independent of the choice of $w$, such that
\[ |w_1| \leq k \min \{ |g|_X : g \in Hw \}.\quad (\dagger) \] 
So the aim of the rest of the proof is
to prove the inequality $(\dagger)$.

Observe that a maximal subword $p_i$ of $w_1$ that is written over $Z_i^*$
for some $i$ may contain symbols from $Z_i \setminus X_i$, but any such
symbols must have arisen from application of the rules in the \eda $R_i$ to
words over $X_i^*$, and so $p_i$ unambiguously represents an element of $H_i$.
Following \cite{GoodmanShapiro}, we decompose $w_1$ as a concatenation
\[ w_1=v_0p_1v_1\cdots p_mv_m, \quad(*) \]
where $p_1$ is defined to be the first subword of $w_1$ (working from the left)
that is written over $Z_{i_1}$ for some $i_1 \in I$,
has maximal length as such a subword, and represents an
element of $H_{i_1}$ of $X_{i_1}$-length greater than $E$.
(Note that the words $v_i$ are denoted by $g_i$ in \cite{GoodmanShapiro}.)
The subwords $p_i$ for $i>1$ are defined correspondingly with respect to the
suffix remaining after removing the prefix $v_0p_1v_1\cdots p_{i-1}$ from $w_1$.

Then any (maximal) subword of any $v_j$ that is written over any $Z_i$
represents an element of $H_i$ of $X_i$-length at
most $E \leq D$, and so $\P(D,E)\,(3)$ ensures that the subword is
written over $X_i$ and is an $X_i$-geodesic. Hence 
$v_j$ is a word over $X$ with no parabolic shortenings.
Then (since $E>2$) property $\P(D,E)\,(4)$ ensures that $v_j$
is also a 2-local geodesic. Also, since $\Phi$ is a set of non-geodesic
words over $X$ of length at most $E$,
$v_j$ cannot contain any subword in $\Phi$. It follows by
Property~\ref{prop:AC2} that each $\widehat{v_j}$ is a 2-local geodesic.

Now, for each $j$, choose $q_j$ to be a geodesic word over $X_{i_j}$ that
represents the same element $h_j$ of $H_{i_j}$ as $p_j$
(so, by Property \ref{prop:AC2}, $q_j$ is also an $X$-geodesic).
Define \[ w_2 = v_0q_1v_1\cdots q_mv_m.\]
We already observed that each $v_j$ is written over $X$, and hence so is $w_2$.
Now according to \cite[Lemma 23]{GoodmanShapiro}, 
the $X$-lengths of non-identity elements of $H_i$ are bounded below by an
exponential function on the lengths of words over $Z_i$ that are their
reductions by the \eda for $H_i$.
So there is certainly a
positive constant $k_1$ 
such that $|q_j| \ge k_1 |p_{i_j}|$ for all $j$,
and hence $|w_2| \ge k_1|w_1|$.
Hence it is sufficient to prove the inequality $(\dagger)$ above for the word
$w_2$ rather than $w_1$,
that is, for some $k'$, show that
\[ |w_2| \leq k' \min \{ |g|_X : g \in Hw \}.\quad (\dagger\!\dagger) \] 

With this in mind, our next step is to construct a word $\widetilde{w_2}$
over $\widehat{X}$, representing the same element of $G$ as $w_2$,
and for which we can use Property~\ref{prop:AC1}.
We define
\[\widetilde{w_2} = \widehat{v_0}h_1\widehat{v_1}\cdots h_m\widehat{v_m}.\]
Note that we would have $\widetilde{w_2}=\widehat{w_2}$ if the subwords $q_j$
were parabolic segments of $w_2$, but this might not be true if the
first generator in some $q_j$ were in more than one parabolic subgroup,
and then we could not be sure that $\widehat{w_2}$ would satisfy the required
conditions. We call the process of conversion of $w_2$ to $\widetilde{w_2}$
{\em modified compression} and we call the subwords $q_j$ of $w_2$
together with the parabolic segments of the subwords $v_j$ the
{\em modified parabolic segments} of $w_2$. Observe that these
modified parabolic segments are all $X$-geodesics.

We want to apply Property~\ref{prop:AC1} to $\widetilde{w_2}$, so we must show
first that $\widetilde{w_2}$ is a 2-local geodesic over $\widehat{X}$.
If not, then $\widetilde{w_2}$ has a non-geodesic subword
$\zeta$ of length 2, equal in $G$ to an $\widehat{X}$-geodesic word $\eta$ of
length at most $1$.  We saw earlier that the subwords $\widehat{v_j}$ are
2-local geodesics, so $\zeta$ must contain some $h_j$; that is, $\zeta=yh_j$ or
$\zeta = h_jy$, where $y \in \widehat{X}$.
Then the definition of the $p_j$ as maximal $Z_{i_j}$-subwords of $w_1$
ensures that $y \not \in H_{i_j}$, so $|\eta| = 1$. 
So, since
$\lambda,c+1 \ge 1$, $\zeta$ and $\eta$ are both $(\lambda,c+1)$-quasigeodesics.
But now, since $h_j$ is a component in $\zeta$ of $X$-length greater
than $E>\epsilon = \epsilon(\lambda,c+1,0)$,
we can apply Property~\ref{prop:BCP} to the paths in $\HatCay$ labelled
by $\zeta$ and $\eta$, and deduce that $\eta$ contains a component connected
to the component $h_j$, and so $\eta$ represents an
element of $H_{i_j}$; hence $\zeta \in H_{i_j}$, and we have a contradiction.
(Alternatively, we could apply \cite[Lemma 4.2]{AntolinCiobanu} to deduce
that $|\zeta|=2$, a contradiction.)

To verify the second requirement of Property~\ref{prop:AC1}, we need to check
that $\widetilde{w_2}$ contains no subword in $\Psi$.
So suppose that $\xi$ is such a subword in $\Psi$.
Then $\xi$ cannot contain any of the generators $h_j$, since $h_j$ would
then be a component in $\xi$ of $X$-length greater than $E$, by the conditions
imposed on the decomposition $(*)$ of $w_1$; but this contradicts the
choice of $E$ earlier in this proof
to be greater than the $X$-length of any component of any word in $\Psi$.
So $\xi$ must be a subword of some $\widehat{v_j}$; but in that case, 
$\xi=\widehat{u}$ for some subword $u$ of $v_j$.
We saw earlier that $v_j$ has no parabolic shortenings, and hence
neither does $u$.
So from the definition of $\Phi_2$ we have $u \in \Phi_2 \subseteq \Phi$.
But we also observed earlier that $v_j$ has no subword in $\Phi$, so 
we have a contradiction.

It now follows using Property \ref{prop:AC1} that $\widetilde{w_2}$ is a
$(\lambda,c)$-quasigeodesic over $\widehat{X}$ without vertex-backtracking.

Let $w_3$ be a geodesic over $X$ that represents
an element of $Hw_2 = Hw$; since it is geodesic, $w_3$ cannot contain any
subwords in $\Phi$, and it cannot have any parabolic shortenings. So we can
apply Property~\ref{prop:AC2} to deduce that $\widehat{w_3}$ is a 2-local
geodesic over $\widehat{X}$  and a $(\lambda,c)$-quasigeodesic without
vertex-backtracking.

Now $\widetilde{w_2} =_G h\widehat{w_3}$ for some $h \in H$
and so, since $\widetilde{w_2}$ is a $(\lambda,c)$-quasigeodesic,
\[ |\widetilde{w_2}| \leq \lambda |h\widehat{w_3}|+c. \]

We note that $\widetilde{w_2}$ and $h\widehat{w_3}$ are both
$(\lambda,c+1)$-quasigeodesics over $\widehat{X}$ without vertex-backtracking,
and the initial and terminal vertices of the paths in $\HatCay$ that
they label coincide.  So we have Properties \ref{prop:BCP}\,(2) and (3)
concerning the components of the two paths.
(This is why we chose $\epsilon =\epsilon(\lambda,c+1,0)$ at the beginning
of the proof.)

We choose a geodesic word $w_h$ over $X_0$ that represents $h$, 
and consider the words $w_2$ and $w_hw_3$.
We want to compare the lengths of $w_2$ and $w_hw_3$, and we do this by
examining the processes of (modified) compression of $w_2$ and $w_hw_3$ to
$\widetilde{w_2}$ and $\widehat{w_hw_3}=h\widehat{w_3}$. 
The word $w_2$ can be decomposed as a concatenation of
disjoint subwords that are its {\em long} modified parabolic segments, its
{\em short} modified parabolic segments and its maximal subwords over $X
\setminus X_I$; we define a modified parabolic segment to be {\em long} if its
length is greater than $4 \epsilon$ and {\em short} otherwise.

Now modified compression reduces the total length of subwords of the second
type, which are replaced by single elements of $\widehat{X}$, by a
factor of at most $4\epsilon$, while the subwords of the third type are
unchanged.
So the total length of the subwords of $w_2$ of the second and third
types is bounded by 
\[ 4\epsilon | \widetilde{w_2}| \leq 4 \epsilon(\lambda
|h\widehat{w_3}| +c) \leq 4 \epsilon (\lambda | w_hw_3| + c),\]
where the second equality follows from the fact that $h\widehat{w_3}
= \widehat{w_hw_3}$.

Now let $u$ be a long modified parabolic segment of $w_2$. 
Then Property \ref{prop:BCP} applied to $\widetilde{w_2}$ and $h\widehat{w_3}$ 
ensures that there is a corresponding parabolic segment $u'$ of
$w_hw_3$ such that $\widehat{u}$ and $\widehat{u'}$ are connected components
of length $1$ of $\widetilde{w_2}$ and $h\widehat{w_3}$
Then since $\widehat{u}$ and $\widehat{u'}$ must be $\epsilon$-similar, it
follows that the initial and terminal points of $u'$ must be
within $X$-distance $\epsilon$ of the initial and terminal points
(respectively) of $u$.
Then $4\epsilon < |u| \leq 2\epsilon + |u'|$, and so $|u| \leq 2|u'|$
(see Fig. \ref{fig:BCP}).
\begin{figure} 
\begin{center}
\setlength{\unitlength}{1.0pt}%
\begin{picture}(200,120)(-79,-60)%
\put(-140,60){$\widetilde{w_2}$ (quasi-geodesic)}%
\put(-145,0){\circle*{5}}%
\put(-147,-13){$\IdHatCay$}%
\put(-140,-60){$h\widehat{w_3}$ (quasi-geodesic)}%
\put(5,58){component $\hat{u}$, $|u|_X > 4\epsilon$}%
\put(30,-55){$\widehat{u'}$}%
\put(-145,0){\line(2,1){100}}
\put(-45,50){\line(1,0){145}}
\put(100,50){\line(3,-1){60}}
\put(160,30){\line(1,-1){30}}
\put(-145,0){\line(3,-2){60}}
\put(-85,-40){\line(1,0){195}}
\put(110,-40){\line(2,1){80}}
\multiput(0,-40)(3.32,19.95){5}{\line(1,6){1.7}}
\multiput(70,-40)(6.65,19.95){5}{\line(1,3){3.4}}
\put(15,3){$\le \varepsilon$}
\put(95,3){$\le \varepsilon$}
{\thicklines
\put(15,50){\line(1,0){85}}
\put(0,-40){\line(1,0){70}}
}
\end{picture}%
\end{center}%
\caption{The paths $\widetilde{w_2}$ and $h\widehat{w_3}$ in
$\HatCay$: bounded coset penetration}
\label{fig:BCP}
\end{figure}
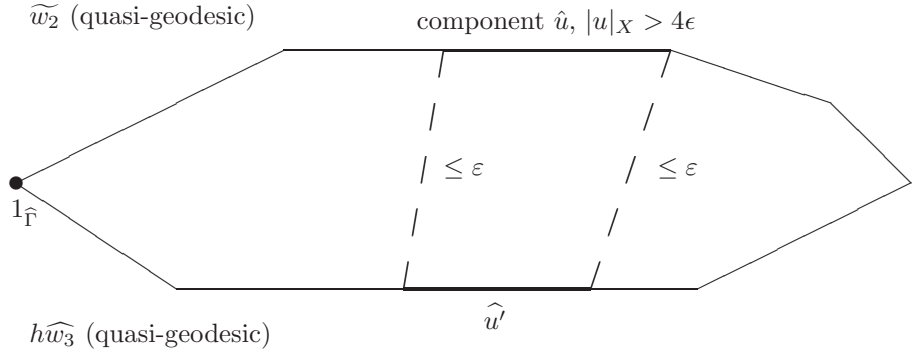

Since $\widetilde{w_2}$ does not backtrack, distinct components of
$\widetilde{w_2}$ must correspond to distinct components of $h\widehat{w_3}$,
and we deduce that the total length of the long
parabolic segments in $w_2$ is bounded above by $2|w_hw_3|$. So
\[ |w_2| \leq (4 \epsilon \lambda +2) | w_hw_3| + 4\epsilon c.\]

Now we consider $w_h$, which is a parabolic segment of $w_hw_3$. 
If $|w_h| > \epsilon$ then Property \ref{prop:BCP} applied to (the paths
labelled by) $h\widehat{w_3}$ and $\widetilde{w_2}$ ensures the existence of
a corresponding parabolic segment $u_2$ in $w_2$, whose initial vertex is within
$X$-distance $\epsilon$ of the basepoint $\IdHatCay$ of the Cayley graph
$\HatCay$, and whose terminal vertex must be in $H$. Let $w_2$
factorise as a concatenation of subwords $u_1u_2u_3$. 
Since $u_1u_2$ is a prefix of $w_2$ that represents an element of $H$,
the fact that $\widetilde{w_2}$ does not vertex-backtrack ensures that its
subword $\widehat{u_1}\widehat{u_2}$ must have length at most 1, and hence
$u_1$ is empty. But now the prefix $u_2$ of $w_2$ is written over the
generators of $H$, and $w_2$ cannot have a non-trivial such prefix, since
$w_1$ (from which it was derived) was reduced by the \eda $S$, and we have
a contradiction.

So now $|w_h| \leq \epsilon$, and we can deduce from the inequality
above that
\[ |w_2| \leq A |w_3| + B \] for some constants $A,B$.
Provided that $|w_3| \neq 0$, it follows that
\[ |w_2| \leq (A+B)|w_3|. \] 
But if $|w_3|=0$, then $w_2$ must represent an element of $H$ and so, since
$\widetilde{w_2}$ has already been proved not to vertex-backtrack,
$\widetilde{w_2}$ must be a word written over $H$ of length at most 1. It
follows that $w_2$ is a word over $X_0$, and so, just as above, we deduce that
$|w_2|=0$, and so the same inequality holds.  This
completes our verification of the condition of $(\dagger\!\dagger)$ (and hence $(\dagger)$),
and the theorem is proved.


\section{The proof of Theorem \ref{thm:qcsubhyp} (i)}
\label{sec:qcthmi}
Let $H$ be a quasiconvex and almost malnormal subgroup of a hyperbolic group
$G = \langle X \rangle$. It is observed in
\cite[Section 1, Example (III)]{Osin06} that $G$ is hyperbolic relative to
$\{ H \}$ so we can apply the results of Section \ref{sec:relhypgps}. 
The proof of Theorem \ref{thm:qcsubhyp} (i) is very similar to that
of Theorem \ref{thm:relhyp} (although we are no longer assuming that $H$ is
virtually abelian) but is more straightforward, so we shall only
summarise it here. In particular, we have $Z=X$ so the complications arising
from the elements of $Z \setminus X$ that do not necessarily represent group
elements do not arise.

We start by extending $X$ as before to include the finite subset $\H'$ of
$\H$  defined in Property \ref{prop:AC2}. Since we are not applying
Proposition \ref{prop:P(D,E)} in this proof, the other adjustment to $X$
is not necessary. We define $\Psi, \Phi, \lambda, c, \epsilon, E$ as
before and put $D=E$.

The standard Dehn algorithm for solving $\WP(G,X)$
consists of all rules $u \to v$ with $u,v \in X^*$ such that $u =_G v$ and
$4\delta \ge |u| > |v|$, where $\delta$ is the `thinness' constant of
$G$ with respect to $X$ (i.e. all geodesic triangles in $\Cay$ are
$\delta$-thin); see \cite[Theorem 2.12]{AL}.
Words $w$ that are reduced by this algorithm are $4\delta$-local geodesics,
and it is proved in \cite[Proposition 2.1]{Holt} that, if $w$ represents
the group element $g$, then $|w| \le 2|g|$.

We define our Dehn algorithm $R$ for $\WP(G,X)$ to consist of all rules $u \to v$ as above,
with $k \ge |u| > |v|$, where $k = \max(2D,4 \delta)$.
Then $R$-reduced words have the property that subwords representing
group elements of $X$-length at most $D$ are $X$-geodesics.
Since $D=E$, it is also true that $R$-reduced words of length at most $E$ are
geodesic, so $R$ has the required property $\P(D,E)$.
As in the previous proof, we define $S$ to be the \eda for
$\GWP(G,H,X)$ consisting of $R$ together with rules $Hx \to H$ for all
$x \in X_0 := X \cap H$.

As before, we suppose that $S$ reduces the input word $Hw$ to $Hw_1$ and
define the decomposition $(*)$ of $w_1$ with $p_i$ being maximal $X_0$-subwords
of $w_1$ that represent group elements of $X$-length greater than $E$.
Again we let $q_i$ be geodesic words over $X_0$ (and hence also over $X$)
with $q_i =_G p_i$. By \cite[Proposition 2.1]{Holt}, we have
$|q_i| \ge k_1|p_i|$ with $k_1 = 1/2$, so again we have
$|w_2| \ge k_1|w_1|$. The remainder of the proof is identical to that
of Theorem \ref{thm:relhyp}. 

\section{Proof of Theorem \ref{thm:qcsubhyp} (ii)}\label{sec:qcthmii}
As we did for Part (i) of this theorem, we prove Theorem \ref{thm:qcsubhyp}
(ii) by constructing an \eda over the alphabet $X \cup \{ H \}$,
where $G = \langle X \rangle$. As in the two earlier proofs, we verify that
the conditions of Proposition~\ref{prop:realtime} hold, to complete the proof.

For a finite (inverse-closed) generating set $X$ of an arbitrary group $G$, we
define an {\em $X$-graph} to be a graph with directed edges labelled by
elements of $X$, in which, for each vertex $p$ and each $x \in X$, there is a
single edge labelled $x$ with source $p$ and, if this edge has target $q$,
then there is an edge labelled $x^{-1}$ from $q$ to $p$. So the {\em Cayley
graph} $\Cay(G,X)$ and, for a subgroup $H \le G$, the {\em Schreier graph}
$\Sch(G,H,X)$ of $G$ with respect to $H$ are examples of $X$-graphs. 
We shall denote the base points of the Cayley and Schreier graphs by $\IdC$ and
$\IdS$ respectively.

Following \cite[Chapter 4]{Foord}, for $k \in \N$, we define the condition
$\GIB(k)$ (which stands for \emph{group isomorphic balls}) for $\Sch:=\Sch(G,H,X)$ as follows.
\begin{quote}
$\GIB(k)$: there exists $K \in \N$
such that, for any vertex $p$ of $\Sch$ with $d(\IdS, p) \ge K$,
the closed $k$-ball $B_k(p)$ of $\Sch$ is $X$-graph isomorphic to the $k$-ball
$B_k(\IdC)$ of $\Cay(G,X)$.
\end{quote}
We say that $\Sch$ satisfies
$\GIB(\infty)$ if it satisfies $\GIB(k)$ for all $k \ge 0$.
The following result is proved in \cite[Theorem 4.3.1.1]{Foord}.
Since its proof may not be readily available, we shall sketch it in
Section \ref{sec:foord}.

\begin{proposition}\label{prop:foord}
Let $H$ be a quasiconvex subgroup of the hyperbolic group $G$. Then
$\Sch(G,H,X)$ satisfies $\GIB(\infty)$ if and only if,
for all $1 \ne h \in H$, the index
$|C_G(h):C_H(h)|$ is finite.
\end{proposition}

Suppose that $G,H$ satisfy the hypotheses of Theorem~\ref{thm:qcsubhyp} (ii).
So $\Sch := \Sch(G,H,X)$ satisfies $\GIB(\infty)$ and, by \cite[Theorem
4.1.3.3]{Foord} or \cite{Kapovich}, $\Sch$ is $\delta$-hyperbolic for some
$\delta>0$ (that is, geodesic triangles in $\Sch$ are $\delta$-thin).

Let $k$ be an integer with $k \geq 4\delta$.
Let $K$ be an integer that satisfies the condition in the definition of
$\GIB(k)$, and let $R=2K$. We can assume that $K \ge \max(k,2)$.
We define our \eda to consist of all rules of the following two forms:
\begin{eqnarray*}
Hv_1 \rightarrow Hv_2,&& |v_2|<|v_1| \leq R,\quad(1) \\
u_1 \rightarrow u_2,&& |u_2| < |u_1|\leq k,\quad(2)
\end{eqnarray*}
where $v_1,v_2,u_1,u_2 \in X^*, v_1v_2^{-1} \in H, u_1=_G u_2$.

In order to apply Proposition~\ref{prop:realtime}
we need to verify that,
whenever $w \in X^*$ and $Hw$ is reduced according to the above \eda,
the length of the shortest string $v$ over $X$ with $Hv=Hw$ is bounded below
by a linear function of $|w|$. 

We shall use \cite[Proposition 2.1]{Holt}: if $u$ (of length $>1$) is a
$k$-local geodesic in a $\delta$-hyperbolic graph, with $k\geq 4\delta$, then
the distance between the endpoints of $u$ is at least $|u|/2 + 1$.

So suppose that $Hw$ is reduced according to the \eda.
If $w$ has length at most $R$, then since $Hw$ is reduced by rules of type (1),
$Hw$ is geodesic in $\Sch$, and
the inequality in Proposition~\ref{prop:realtime} holds with $k=1$.

So suppose that $|w|>R$, and let $w_1$ be the prefix of length $R$ of
$w$.  We aim to show that every vertex of $\Sch$ that comes after $w_1$
on the path from $\IdS$ labelled $w$ lies outside of $B_K(\IdS)$.
Choose $w_2$ so that $w_1w_2$ is maximal as a prefix of $w$ subject to all
vertices of $w_2$ lying outside of $B_K(\IdS)$.
Then, since $w_1$ is geodesic of length $R=2K$, we have $|w_2| \geq K-1$,
and so $|w_1w_2| \geq 3K-1$. Since $w_1$ is geodesic in $\Sch$, and
that part of the
path labelled $w_1w_2$ that lies outside of the $K$-ball is a $k$-local
geodesic in $\Sch$ (because, by $\GIB(k)$, it is isometric to the corresponding
word in the Cayley graph, and
our inclusion of the rules of type (2) in the \eda ensures that
the reduced words over $X$ of length $\leq k$ are geodesics),
we see that the whole of the path labelled $w_1w_2$ is a $k$-local geodesic
in $\Sch$.
So we can apply \cite[Proposition 2.1]{Holt} to deduce from the
$\delta$-hyperbolicity of $\Sch$ that 
\[ d_{\Sch}(\IdS,Hw_1w_2) \geq |w_1w_2|/2 + 1\geq (3K+1)/2>K+1.\]
It follows that, if $w_1w_2$ were not already equal to $w$, then     
it would be extendible to a longer prefix of $w$ subject to all
vertices of $w_2$ lying outside of $B_K(\IdS)$.
So $w_1w_2=w$, and the above inequality gives us the linear lower bound
$d_{\Sch}(\IdS,Hw) \ge |w|/2+1 \ge |w|/2$ on $d_{\Sch}(\IdS,Hw)$.
So the inequality in Proposition~\ref{prop:realtime} holds with $k=2$
and hence, by the preceding paragraph, it holds with $k=2$ for all
words $w$ such that $Hw$ is reduced according to the \eda.
The result now follows from Proposition~\ref{prop:realtime}, and this
completes the proof of Theorem~\ref{thm:qcsubhyp} (ii).

\section{Proof of Theorem \ref{thm:subvf}} \label{sec:subvf}
Let $F$ be a free subgroup of $G = \langle X \rangle$ with $|G:F|$ finite,
and let $Y$ be the inverse closure of a free generating set for $F$;
that is the union of a free generating set with its inverses.
Let $K=F \cap H$.
The subgroup $K$ has finite index in $H$, 
and so must (like $H$) be finitely generated.
It easy to see that the elements in any right transversal of $K$ in $H$ lie
in different cosets of $F$ in $G$, so we can extend a right transversal
$T'=\{t_1,\ldots,t_m\}$ of $K$ in $H$ to a right transversal
$T = \{t_1,\ldots,t_n\}$ of $F$ in $G$. 
Then any word $w \in X^*$ can be expressed (in $G$) as a word in $U^*T$,
where $U$ is the set $\{ u(i,x) : 1 \le i \le n,\,x \in X \}$ of Schreier
generators for $F$ in $G$ defined by the equations $t_i x= u(i,x)t_j$.
(Note that $X$ inverse-closed implies that $U$ is inverse-closed.)
Then, by substituting the reduced word in $Y^*$ for each $u(i,x)$, the word
$w$ can be written as a word $vt$ in $Y^*T$
(where $v$ is not necessarily freely reduced).

The first step to recognise whether $w \in H$ is to rewrite it to the form
$vt$, as above, using a transducer. Then $w \in H$ if and only if $t \in
T'$ and $v \in K$.  It remains for us to describe the
operation of a deterministic pushdown automaton (\pda) $N$ to recognise those
words $v$ in $Y^*$ that lie in the subgroup $K$ of the free group $F$.
Note that this machine operates simultaneously, rather than sequentially, with
the transducer, and it follows from the fact that context-free languages are
closed under inverse {\gsm}s \cite[Example 11.1, Theorem 11.2]{HopcroftUllman}
that the combination of the two machines is a \pda.

By \cite{AnS75} or \cite[Proposition 4.1]{GerstenShort91b}, any finitely
generated subgroup $K$ of a free group is {\em $L$-rational}, where $L$ is the
set of freely reduced words over a free generating set; that is, the set
$K \cap L$ is a regular language.
In our case, we choose $L$ to ber the freely reduced words over $Y$.

We shall build our \pda $N$ out of a finite state automaton (\fsa) $M$ for which
$L(M) \cap L = K \cap L$ (where $L(M)$ is the set of words accepted by $M$).
The construction is (in effect) described in the proof of
\cite[Theorem 2.2]{GerstenShort91b} that $K$ is $L$-quasiconvex. 
The $L$-quasiconvexity condition is equivalent to the property that
all prefixes of freely reduced words that represent elements of $K$ lie within a
bounded distance of $K$ in the Schreier graph $\Sch := \Sch(F,K,Y)$.
Equivalently, a freely reduced  word $v$ over $Y$ represents an element
of $K$ precisely if it labels a loop in $\Sch$ from $\IdS$ to $\IdS$
that does not leave a particular bounded neighbourhood $B=B_d(\IdS)$ of $\IdS$.

Suppose that $g_1=1$ and that $K=Kg_1,Kg_2,\ldots,Kg_r$ are the right cosets 
corresponding to the bounded neighbourhood $B$ of $\IdS$ within $\Sch$ that is
identified above. 
The \fsa $M$ is defined as follows.
\begin{mylist}
\item[(i)] The states of $M$ are denoted by $\sigma_1,\ldots,\sigma_r,
\hat{\sigma}$.
\item[(ii)] The states $\sigma_1,\ldots,\sigma_r$ correspond to the
right cosets $Kg_i$ of $K$ in $F$, where each $g_i$ is in the
finite subset $B$ identified above; indeed we may identify $\sigma_i$ with
the coset $Kg_i$, and then use the name $B$ both for the set
$\{Kg_1,\ldots,Kg_r\}$ of cosets and for the set
$\{\sigma_1,\sigma_2,\ldots,\sigma_r\}$ of states. The state $\sigma_1$ (which
corresponds to the subgroup $K$) is the start state
and the single accepting state.
\item[(iii)] For $1\le i,j \le r$ and $y \in Y$, there is a transition
$\sigma_i^y = \sigma_j$ if and only if $Kg_iy=Kg_j$. It follows from this that
$\sigma_i^y = \sigma_j$ if and only if $\sigma_j^{y^{-1}} = \sigma_i$.
\item[(iv)] $\hat{\sigma}$ is a failure
state, and is the target of all transitions that are not
defined in (iii), including those from $\hat{\sigma}$.
\end{mylist}

We see that, as a word is read by $M$, the automaton keeps track of
the coset of $\Sch$ that contains $Kw$, where $w$ is the prefix that has been
read so far, so long as that coset is within the finite neighbourhood $B$ of 
$\IdS$,
and in addition so are all cosets $Kw'$ for which $w'$ is a prefix of $w$.
The $L$-quasiconvexity of $K$ ensures that
a word in $L$ is accepted by $M$ if and only
if it represents an element of $K$.
In fact any word over $Y$ that is accepted by $M$ must represent an element of
$K$.

However words over $Y$ that are not freely-reduced (that is, not in $L$)
and do not stay inside of $B$ will be rejected
by $M$, even when they represent elements of $K$.
In order to construct a machine that accepts all words $v$ over
$Y$ within $K$, and not simply those that are also freely-reduced,
we need to combine the operation of the \fsa $M$ above with a stack,
which we use to compute the free reduction.

We construct our
\pda $N$ to have the same state set $B \cup \{\hat{\sigma}\}$ as 
$M$, again with $\sigma_1=Kg_1=K$ as the start state and sole accepting state.
The transitions from the states $\sigma_i=Kg_i$ are as in $M$. We need however
to describe the operation of the stack, and transitions from the state
$\hat{\sigma}$, which is non-accepting, but no longer a failure state.

The stack alphabet is the set $Y \cup (Y \times B)$.  The second component of
an element of $Y \times B$ is used to record the state $M$ is in immediately
before it enters the state $\hat{\sigma}$.
In addition, we use the stack to store the free reduction of
the prefix of $v$ that has been read so far.

The operations of the \pda $N$ that correspond to the various transitions of
$M$ are described in the following table.  The absence of an entry in the
`push' column indicates that nothing is pushed.

\begin{center}
\begin{tabular}{|l| lllll|}
\hline
Transition & Input & Input & Pop & Push & Output \\ 
of $M$ & state & symbol & & &  state \\ 
\hline
$\sigma_i^y = \sigma_j$ & $\sigma_i$ & $y$ &  $y^{-1}$ &  & $\sigma_j$ \\
                   & $\sigma_i$ & $y$ & $y' \neq y^{-1}$ & $y'y$ & $\sigma_j$ \\
\hline
$\sigma_i^y = \hat{\sigma}$ & $\sigma_i$ & $y$ & $y'$ &
   $y'(y,\sigma_i)$ & $\hat{\sigma}$ \\
\hline
$\hat{\sigma}^y = \hat{\sigma}$ & $\hat{\sigma}$ &
   $y$ & $(y^{-1},\sigma_i)$ & & $\sigma_i$ \\
& $\hat{\sigma}$ & $y$ & $(y',\sigma_i),\quad y' \neq y^{-1}$ &
   $(y',\sigma_i)y$ & $\hat{\sigma}$ \\
& $\hat{\sigma}$ & $y$ & $y^{-1}$ & & $\hat{\sigma}$ \\
& $\hat{\sigma}$ & $y$ & $y' \neq y^{-1}$ & $y'y$ & $\hat{\sigma}$ \\
\hline
\end{tabular}
\end{center}

Note that we have not specified that $y'\neq y^{-1}$ in line 3, but in fact the
condition $y'=y^{-1}$ does not arise in this situation.
Since there can be a symbol in $Y \times B$  on the stack only when $N$ is
in state $\hat{\sigma}$, it is not possible to pop such a symbol when
$N$ is in state $\sigma_i$, so there are no such entries in the table.

The fact that $N$ recognises $\GWP(F,K,Y)$ follows from the fact that $N$
accepts $w$ if and only if $M$ accepts $\overline{w}$, where $\overline{w}$ is
the free reduction (in $L$) of $w$.  We prove this by induction on the number
$k$ of reductions of the form $w_1yy^{-1}w_2 \rightarrow w_1w_2$ with $y \in
Y$ that we need to apply to reduce $w$ to $\overline{w}$.

The case $k=0$ of our induction follows from the fact that $L(M)\cap L=K\cap L$,
combined with the observation that, if $w$ is freely reduced, then
$w$ leads to the same state of $M$ as it does of $N$. For in that case
the only possible transitions as we read $w$ are of the types
described in lines 2,3,5,7 of the table.

For $k>0$ it is enough to prove the statement
\begin{quote}
$(*)$:  if $wy \in L$, then the configuration of $N$ after reading $wyy^{-1}$
is identical to the configuration after reading $w$.
\end{quote}
It follows from $(*)$ that a word $w_1yy^{-1}w_2$ in which $yy^{-1}$
is the leftmost cancelling pair is accepted by $N$ if and only if $w_1w_2$ is accepted by
$N$, and hence we have the inductive step we need.

We can check the statement $(*)$ with reference to the table.  There
are up to seven possibilities for the type of transition of $N$ as the final
symbol $y$ of $wy$ is read.

For the first two of these,
$N$ is in state $\sigma_i$ after reading $w$,
and moves to a state $\sigma_j$.
Since $wy$ is in $L$, the top stack symbol after reading $w$ is not $y^{-1}$.
Hence the transition must be of the type described in line 2 of the table,
and not as in line 1, that is, $y' \neq y^{-1}$ is popped, and then $y'y$ is
pushed. 
Recalling that $\sigma_i^y = \sigma_j$ in $M$ if and only if
$\sigma_j^{y^{-1}}=\sigma_i$, we see that the next transition of $N$,
from $wy$ on $y^{-1}$, is of the type described in line 1. Then the symbol
$y$ is popped, the symbol $y'$ is again on the top of the stack,
and $N$ returns to the state $\sigma_i$

We consider similarly the remaining five possibilities for the transition from
$w$ on $y'$, and the subsequent transitions on $y'^{-1}$, and verify $(*)$ for
each of those configurations.  
This completes the proof of Theorem~\ref{thm:subvf}.

\section{Proof of Theorem \ref{thm:subvf_conv}} \label{sec:subvf_conv}
Our proof of Theorem \ref{thm:subvf_conv} has the same structure as the proof
in \cite{MullerSchupp83} that groups with $\WP(G)$ context-free are virtually
free, and it would be helpful for the reader to be familiar with that proof.

We shall prove that $G$ has more than one end. Assuming that to be true,
we use Stalling's theorem \cite{Stallings71} to conclude that $G$ has
a decomposition as an amalgamated free product $G = G_1 *_K G_2$, or
as an HNN-extension $G = G_1*_{K,t}$, over a finite subgroup $K$.
Since $G_1$ (and $G_2$) are easily seen to be quasiconvex
subgroups of $G$, it is not hard to show that the hypotheses of the
theorem are inherited by the subgroup $H \cap G_1$ of $G_1$ (and
$H \cap G_2$ of $G_2$), and so they too have more than
one end, and we can apply the Dunwoody accessibility result to
conclude that $G$ is virtually free.

So we just need to prove that $G$ has more than one end.
Fix a finite inverse-closed generating set $X$ of $G$.  Then, as in
\cite{MullerSchupp83}, we consider a context-free grammar in Chomsky normal
form with no useless variables that derives $\GWP(G,H,X)$.
More precisely, we suppose that each rule has the form $\start \to \nullstring$
(where $\start$ is the start symbol), $z \rightarrow z'z''$ or
$z \rightarrow a$, where $z,z',z''$ are variables, and $a$ is terminal, and we
assume that $\start$ does not occur on the right hand side of any
derivation. When a word $w'$ can be derived from a word $w$ by application of
a single grammatical rule we write $w \Rightarrow w'$, and when a
sequence of such rules is needed we write $w \Rightarrow ^* w'$.

Let $z_1,\ldots,z_n$ be the variables of the grammar other than $\start$
and, for each $z_i$ let $u_i$ be a shortest word in $X^*$ with
$z_i \Rightarrow^* u_i$. Let $L$ be the maximum length of the words $u_i$.

Let $w \in \GWP(G,H,X)$ with $|w|>3$, and fix a derivation of $w$ in the
grammar. We shall define a planar $X$-graph $\Delta$ with an associated
$X$-graph homomorphism $\phi:\Delta \to \Sch:= \Sch(G,H,X)$.  We start with a simple
plane polygon with a base point, and edges labelled by the letters of $w$, and
with $\phi$ mapping the base-point of $\Delta$ to $\IdS$. Note that $\phi$ is
not necessarily injective.

If $z_i$ occurs in the chosen derivation of $w$, then we have $w = vv_iv'$
with $z_i \Rightarrow^* v_i$; two such words $v_i$ and $v_j$ are either
disjoint as subwords of $w$ or related by containment.
Since $z_i \Rightarrow^* u_i$, we also have $vu_iv' \in \GWP(G,H,X)$.  So we
can draw a chord labelled $u_i$ in the interior of $\Delta$ between the two
ends of the subpath labelled $v_i$, and $\phi$ extends to this extension of
$\Delta$.
If we do this for each such $z_i$ for which $1 < |v_i| < |w|-1$ then, as in
\cite[Theorem 1]{MullerSchupp83}, we get a `diagonal triangulation' of $\Delta$,
in which the sides are either boundary edges of $\Delta$ or internal
chords of length at most $L$.
(But note that, for the first derivation $\start \to z_1z_2$, say,
if $|v_1|>1$ and $|v_2|>1$ then, to avoid an internal bigon, we
omit the chord labelled $u_2$.)

Suppose, for a contradiction, that $G$ has just one end; that is, for any
$R$, the complement in $\Cay(G,X)$ of any ball of radius $R$ is connected.
Then, for any $R$,
we can find a word $w_1w_2w_3$ over $X$ with $w_1w_2w_3=_G 1$ which, starting
at $\IdC$, labels a simple closed path in
$\Cay(G,X)$, where $|w_1|=|w_3|=R$, $w_3w_1$ is geodesic, and no vertex in
the path labelled $w_2$ is at distance less than $R$ from $\IdC$.

Choose such a path with $R=3L+1$. Choose $k'$ such that the whole path
lies in the ball $B_{k'}(\IdC)$ of $\Cay(G,X)$, and let $k = k'+L$. Then,
since by Proposition \ref{prop:foord} $\Sch(G,H,X)$ satisfies $\GIB(k)$, there
exists $K$ 
such that, for any vertex $p$ of  $\Sch(G,H,X)$ with
$d(\IdS,p) \ge K$, the ball $B_k(p)$ of $\Sch(G,H,X)$ is $X$-graph isomorphic
to the ball $B_k(\IdC)$ of $\Cay(G,X)$.
Choose such a vertex $p$, and consider the path labelled $w_1w_2w_3$ of
$\Sch(G,H,X)$ that is based at $p$, as in Fig. \ref{fig:triangulation}.

Choose a vertex $q$ on the path labelled $w_1w_2w_3$ with $d(\IdS,q)$ minimal,
and let $w_4$ be the label of a geodesic path in $\Sch(G,H,X)$ from
$\IdS$ to $q$. Then, for some cyclic permutation $w'$ of $w_1w_2w_3$, we have
a closed path in $\Sch(G,H,X)$ based at $\IdS$  and labelled $w_4 w' w_4^{-1}$.

We apply the above triangulation process to a planar $X$-graph $\Delta$
for $w_4 w' w_4^{-1}$.
Since $q$ is the closest vertex to $\IdS$ on the loop labelled by $w'$, and the
path from $\IdS$ to $q$ in the Cayley graph labelled by $w_4$ is geodesic, every vertex on that path is
as close to $q$ as to any other vertex of $w'$, and so any vertex of $w_4$ that can be connected by the image of a chord of $\Delta'$ to a vertex of $w'$ must
be within distance at most $L$ of $q$. Let $r$ be the first such vertex on $w_4$
(as we move from $\IdS$ to $q$), and let $w_5$ be the suffix of $w_4$ that labels
the path along $w_4$ from $r$ to $q$.  Then $|w_5|\leq L$.

So we can derive from our triangulation of $\Delta$ a triangulation
of a planar diagram $\Delta'$ for the word $w_5w'w_5^{-1}$, and there is
an associated $X$-graph homomorphism $\phi'$ that maps this to the corresponding
subpath in $\Sch(G,H,X)$. (Note that the images of $w_5$ and $w_5^{-1}$
under $\phi'$ are equal, but that $\phi'$ is injective when restricted to $w'$.)
By our choice of $k = k' + L$, the image of $\phi'$ lies entirely within
$B_k(p)$, which is $X$-graph isomorphic to $B_k(\IdC)$.
So the distances in $\Sch(G,H,X)$ between vertices in this image are the
same as in any path with the same label in $\Cay(G,X)$.


\begin{figure} 
\begin{center}%
\setlength{\unitlength}{1.0pt}%
\begin{picture}(100,210)(-25,40)%
\put(15,45){\circle*{5}}%
\put(20,45){$\IdS$}%
\put(15,45){\line(1,3){37}}%
\put(33.5,100.5){\vector(1,3){0}}%
\put(25,105){\vector(1,3){17}}%
\put(17.5,97.5){$w_4$}%
\put(22,94){\line(-1,-3){15}}%
\put(52,156){\circle*{5}}%
\put(55,145){$q$}%
\put(46,138){\vector(1,3){0}}%
\put(50,133){$w_5$}%
\put(42,126){\circle*{5}}%
\put(45,115){$r$}%
\put(61,163){\vector(-3,1){31}}%
\put(67,151){\vector(-3,1){0}}%
\put(64,159){$w_3$}%
\put(77.5,157.5){\line(3,-1){14}}%
\put(22,166){\circle*{5}}%
\put(25,155){$p$}%
\put(-14,178){\vector(-3,1){0}}%
\put(-13,182){$w_1$}%
\put(97,141){\line(-3,1){141}}%
\qbezier(97,141)(62,320)(-44,188)%
\put(66,226){\vector(1,-1){0}}%
\put(61,233){$w_2$}%
\end{picture}%
\caption{Triangulation in $\Sch(G,H,X)$}
\label{fig:triangulation}
\end{center}%
\end{figure}

As in \cite{MullerSchupp83}, we colour, using three colours, the vertices of
the  boundary paths of $\Delta'$ that are labelled
$w_1,w_2,w_3$ (where vertices on two of these subwords get both associated
colours), and we colour the vertices on $w_5$ and $w_5^{-1}$ with
the same colour (or colours) as  $q$. As in \cite[Lemma 5]{MullerSchupp83},
we conclude that there is a triangle in the triangulation
whose vertices use all three colours between them. One (or even two) of these
vertices could be on the subpath labelled $w_5$, and two 2-coloured vertices
in the triangle might coincide, but, since any vertex on $w_5$ is within
distance $L$ of $q$, replacing vertices on $w_5$ by $q$ as necessary, we end
up with a triangle of three (not necessarily distinct) vertices 
$p_1, p_2, p_3$  with $p_i$ on $w_i$, and
with $d(p_i,p_j) \le 2L$ for each $i,j$. At least one of $p_1,p_3$
must be within distance $L$ of $p$. But then $d(p,p_2) \leq 3L$, contradicting
our assumption that $w_2$ is outside $B_{3L}(p)$.
This completes the proof of Theorem~\ref{thm:subvf_conv}.

\section{Sketch of proof of Proposition \ref{prop:foord}} \label{sec:foord}
Suppose first that $|C_G(h):C_H(h)|$ is infinite for some $1 \ne h \in H$,
and let $w$ be a word representing $h$.
Then, for any $K>0$, there exists a word $v \in C_G(h)$ labelling a
path in $\Sch:=\Sch(G,H,X)$ from $\IdS$ to a vertex $p$ with
$d(\IdS,p)>K$, and there is a loop labelled $w$ based at $p$ in
$\Sch$, but no such loop based at $\IdC$ in $\Cay:= \Cay(G,X)$. So $\GIB(|w|)$
fails in $\Sch$.

Suppose conversely that $\GIB(k)$ fails in $\Sch$ for some $k$.
Then there are vertices $p$ of $\Sch$ at arbitrarily large distance from
$\IdS$ such that the ball $B_k(p)$ in $\Sigma$ is not $X$-graph isomorphic to
the ball $B_k(\IdC)$ in $\Cay$.
So the natural labelled graph morphism $B_k(\IdC) \to B_k(p)$ with
$\IdC \mapsto p$ is not injective, and hence two distinct vertices
of $B_k(\IdC)$ map to the same vertex of $B_k(p)$. 
So, for any such vertex $p$, there is at least one labelled loop based at $p$, 
within $B_k(p)$, such that the corresponding labelled path based at
in $B_k(\IdC)$ in $B_k(\IdC)$ is not a loop in $\Cay$.
Since the number of words that can label loops in a ball of radius $k$ in
$\Sch$ is finite, some word $w$ with $w \ne_G 1$ must label loops based at $p$
for infinitely many vertices $p$ of $\Sigma$,
and we can choose $w$ to be geodesic over $X$. 
Now for any integer $N$, there is a word $v$ of length greater that $N$,
labelling a geodesic in $\Sch$ from $\IdS$ to a vertex $p$, from which there
is a loop in $\Sch$ labelled by $w$.

For such a word $v$, we have $hvw = v$ for some $h \in H$.  Let $u$ be a
geodesic word labelling $h$. Then we have a geodesic quadrilateral 
with vertices $A=\IdC,B,C,D$ in $\Cay(G,X)$ with sides  
$AB$, $BC$, $CD$, $AD$ labelled $u$, $v$, $w$, $v$, respectively,
as shown in Fig. \ref{fig:ABCD}.

\begin{figure} 
\begin{center}
\setlength{\unitlength}{1.0pt}%
\begin{picture}(150,180)(-70,5)%
\put(50,10){\circle*{5}}%
\put(57,8){$A=1_\Cay$}%
\put(-50,10){\circle*{5}}%
\put(-67,8){$B$}%
\qbezier(50,10)(10,15)(12,50)%
\qbezier(-50,10)(-10,15)(-12,50)%
\qbezier(-50,10)(-15,10)(-5,30)%
\qbezier(50,10)(15,10)(5,30)%
\qbezier(-5,30)(0,35)(5,30)%
\put(-2.5,32.2){\vector(-1,0){0}}%
\put(-3,36){$u$}%

\put(-12,50){\line(0,1){85}}%
\put(-12,87){\vector(0,1){0}}%
\put(-25,82){$v$}%
\put(12,50){\line(0,1){85}}%
\put(12,87){\vector(0,1){0}}%
\put(17,82){$v$}%

\put(-40,160){\circle*{5}}%
\put(-57,158){$C$}%
\put(40,160){\circle*{5}}%
\put(47,158){$D$}%
\qbezier(40,160)(12,155)(12,135)%
\qbezier(-40,160)(-12,155)(-12,135)%
\qbezier(40,160)(5,155)(5,135)%
\qbezier(-40,160)(-5,155)(-5,135)%
\qbezier(-5,135)(0,125)(5,135)%
\put(2,130.2){\vector(1,0){0}}%
\put(-5,121){$w$}%
\end{picture}%
\end{center}%
\caption{The geodesic quadrilateral $ABCD$}
\label{fig:ABCD}
\end{figure}
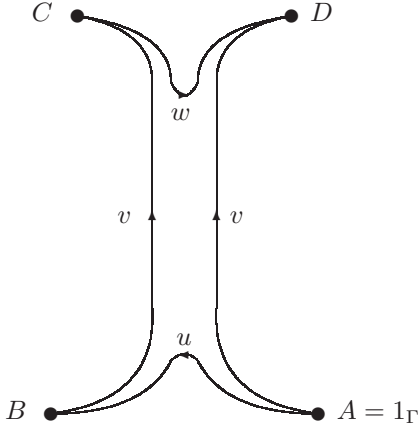


By the hyperbolicity of $G$, each vertex of $AB$ lies within a distance
$2\delta$ of some vertex on $BC$, $CD$ or $DA$, where $\delta$ is the
constant of hyperbolicity. Furthermore, since $H$ is quasiconvex in $G$,
there is a constant $\lambda$, such that each vertex of $AB$ is within a
distance $\lambda$ of a vertex of $\Cay$ representing an element of $H$.
Since each vertex of $w$ lies at distance at least $|v|-k$ from any vertex
in $H$, by choosing $|v| > k+2\delta+\lambda$ we can ensure that none of the
vertices of $AB$ is $2\delta$-close to any vertex of $CD$. So the vertices
of $AB$ must all be $2\delta$-close to vertices in $BC$ or $DA$.
But, since $v$ labels a geodesic path from $\IdS$ in $\Sch$, at most
$2\delta + \lambda$ vertices on $BC$ or on $DA$ can be within $2\delta+\lambda$
of a vertex in $H$. So each vertex of $AB$ is at distance at most $2\delta$
from one of at most $4\delta + 2\lambda$ vertices and, since the total number of
vertices in $\Cay$ with that property is bounded, we see that
$|AB| = |u|$ is bounded by some expression in $|X|$, $\delta$ and $\lambda$.

By hyperbolicity of $G$, the two paths $BC$ and $AD$ labelled $v$
must synchronously $L$-fellow travel for some $L$ (which depends on the upper
bounds on $|w|$ and $|u|$).
Let $m>0$. Then, by choosing $v$ sufficiently long, we can ensure that
some word $u'$ appears as a word-difference between $BC$ and $AD$ at least $m$
times; 
that is, $v$ has consecutive subwords $v_0,\ldots,v_m,v'$, such that
$v = v_0v_1v_2 \cdots v_m v'$, and $hv_0v_1v_2 \cdots v_iu' =_G v_0v_1v_2
\cdots v_i$ for each $i$ with $0 \le i \le m$.
The case $i=0$ gives $u' = v_0^{-1}h^{-1}v_0$, and
it follows from this that
$g_i := v_0(v_1v_2 \cdots v_i)v_0^{-1} \in C_G(h)$ for $1 \le i \le m$.
Also, since $v$ labels a geodesic in $\Sch$, the elements 
$v_0v_1, v_0v_1v_2, \ldots v_0v_1v_2 \cdots v_m$ lie in distinct cosets of
$H$ and hence so do the $g_i$. Since we can choose $m$ arbitrarily
large, this contradicts the finiteness of $|C_G(h):C_H(h)|$.

\section*{Acknowledgements}

The first author was supported by the Swiss National Science 
Foundation grant Professorship FN PP00P2-144681/1, and would like to thank the
mathematics departments of the Universities of Newcastle and Warwick for their
support and hospitality.

\textsc{Laura Ciobanu,
Mathematical and Computer Sciences,
Colin McLaurin Building, 
Heriot-Watt University,      
Edinburgh EH14 4AS, UK}

\emph{E-mail address}{:\;\;}\texttt{l.ciobanu@hw.ac.uk}
\bigskip

\textsc{Derek Holt,
Mathematics Institute,
Zeeman Building,
University of Warwick,
Coventry CV4 7AL, UK
}

\emph{E-mail address}{:\;\;}\texttt{D.F.Holt@warwick.ac.uk}
\bigskip

\textsc{Sarah Rees,
School of Mathematics and Statistics,
University of Newcastle,
Newcastle NE1 7RU, UK
}

\emph{E-mail address}{:\;\;}\texttt{Sarah.Rees@ncl.ac.uk}

\end{document}